\documentclass{amsart}
\usepackage{latexsym}
\usepackage{amstext}
\usepackage{amsmath}
\usepackage{amssymb}
\usepackage{amsthm}
\usepackage{mathrsfs}
\usepackage{url}
\usepackage{pictexwd,dcpic}

\newtheorem{theorem}{Theorem}
\newtheorem{lemma}[theorem]{Lemma}
\newtheorem{corollary}[theorem]{Corollary}

\newtheorem{problem}[theorem]{Problem}

\theoremstyle{definition}
\newtheorem{definition}[theorem]{Definition}

\DeclareMathOperator{\ord}{ord}

\def \O {\mathcal{O}}
\def \pt {\rm{pt}}

\def \kbar {\overline{k}}
\def \Lbar {\overline{L}}

\DeclareMathOperator{\dv}{div}
\DeclareMathOperator{\Jac}{Jac}
\DeclareMathOperator{\Prym}{Prym}
\DeclareMathOperator{\im}{im}
\DeclareMathOperator{\chr}{char}

\DeclareMathOperator{\Div}{Div}
\DeclareMathOperator{\Nm}{Nm}
\DeclareMathOperator{\Norm}{N}
\DeclareMathOperator{\CR}{CR}
\DeclareMathOperator{\GL}{GL}
\DeclareMathOperator{\disc}{Disc}

\begin{document}
\bibliographystyle{amsplain}
\title{Siegel's Theorem and the Shafarevich Conjecture}
\author{Aaron Levin}
\address{Department of Mathematics\\Michigan State University\\East Lansing, MI 48824}
\email{adlevin@math.msu.edu}
\date{}
\begin{abstract}
It is known that in the case of hyperelliptic curves the Shafarevich conjecture can be made effective, i.e., for any number field $k$ and any finite set of places $S$ of $k$, one can effectively compute the set of isomorphism classes of hyperelliptic curves over $k$ with good reduction outside $S$.  We show here that an extension of this result to an effective Shafarevich conjecture for {\it Jacobians} of hyperelliptic curves of genus $g$ would imply an effective version of Siegel's theorem for integral points on hyperelliptic curves of genus $g$.
\end{abstract}

\maketitle

\section{Introduction}

The famous 1929 theorem of Siegel \cite{Sie} on integral points on affine curves states (in a formulation convenient for us):
\begin{theorem}[Siegel]
Let $C$ be a curve over a number field $k$, $S$ a finite set of places of $k$ containing the archimedean places, $\O_{k,S}$ the ring of $S$-integers, and $f\in k(C)$ a nonconstant rational function on $C$.  If $C$ is a rational curve then we assume further that $f$ has at least three distinct poles.  Then the set of $S$-integral points of $C$ with respect to $f$,
\begin{equation*}
C(f,k,S)=\{P\in C(k)\mid f(P)\in \O_{k,S}\},
\end{equation*}
is finite.
\end{theorem}

While Siegel's theorem is completely satisfactory from a qualitative viewpoint, all known proofs of the theorem suffer from the defect of being ineffective, i.e., in general there is no known algorithm for explicitly computing the set $C(f,k,S)$ (when it is finite).  In the classical proofs of Siegel's theorem, this ineffectivity arises from the use of Roth's theorem from Diophantine approximation (for a survey on the use of Roth's theorem in Siegel's theorem, including some remarks on effectivity, see \cite{Zan}).  Finding an effective version of Siegel's theorem remains a longstanding important open problem.

Of course, in certain special cases there are known techniques for effectively computing $C(f,k,S)$.  In this context, the most powerful and widely used effective techniques come from Baker's theory of linear forms in logarithms \cite{Bak}.  Using these techniques, Baker and Coates \cite{BC} proved an effective version of Siegel's theorem for curves of genus zero and genus one.  Already for curves of genus two, however, it is an open problem to prove an effective version of Siegel's theorem.  More generally, we will be interested here in studying this problem for hyperelliptic curves:

\begin{problem}
\label{mp}
Find an effective version of Siegel's theorem for curves $C$ of genus two.  More generally, find an effective version of Siegel's theorem for hyperelliptic curves $C$.
\end{problem}

Again, in certain special cases, Problem \ref{mp} has been solved.  For instance, if $C$ is a (nonsingular projective) hyperelliptic curve over a number field $k$, $i$ denotes the hyperelliptic involution of $C$, and $f\in k(C)$ is a rational function that has a pole at both $P$ and $i(P)$ for some point $P\in C(\kbar)$ (where we allow $P=i(P)$), then $C(f,k,S)$ is effectively computable for any appropriate finite set of places $S$.  This is essentially equivalent to effectively finding all solutions $x,y\in \O_{k,S}$ to hyperelliptic equations of the form
\begin{equation}
\label{hypeq}
y^2=a_nx^n+\cdots+a_0, \quad x,y\in \O_{k,S},
\end{equation}
where $a_0,\ldots, a_n\in k$ are constants and the equation defines a hyperelliptic curve.  Explicit bounds for the solutions to \eqref{hypeq} (when $\O_{k,S}=\mathbb{Z}$) were first given by Baker \cite{Bak2}.  Along the lines of equation \eqref{hypeq}, we note that using Riemann-Roch it is possible to restate Problem \ref{mp} as a question about integral solutions to certain specific types of equations.  For instance, if $C$ has genus two, Problem \ref{mp} is equivalent to solving equations of the form \eqref{hypeq} (with $n=5,6$) and equations of the form
\begin{equation*}
y^3+g_1(x)y^2+g_2(x)y=x^4+g_3(x), \quad x,y\in \O_{k,S},
\end{equation*}
where $g_i\in k[x]$ has degree $\leq i$.  We refer the reader to \cite{Grant} or \cite{Pou} for details.

When $C$ has genus two, Bilu \cite{Bilu} has shown that there exists an infinite set $M\subset C(\kbar)$ such that if $f$ has poles at two distinct points of $M$, then $C(f,k,S)$ is effectively computable.  The method of \cite{Bilu} combines linear forms in logarithms with functional units and coverings of curves, and the exact limitations of the method do not yet seem to be fully understood.  Another interesting alternative approach to an effective Siegel's theorem for genus two curves is given by Grant in \cite{Grant}.  There, the problem is reduced to questions about integral points on a certain affine subset of the Jacobian of $C$, which in turn are reduced to certain ``non-Abelian $S$-unit equations".

A final case of particular note where an effective version of Siegel's theorem is known is the case of geometrically Galois coverings of the projective line, proved independently by Bilu \cite{Bilu2} and Dvornicich and Zannier \cite{DZ} (partial results were also obtained in \cite{Kl} and \cite{Pou}).  This generalizes, in a qualitative way, the aforementioned results of Baker and Coates \cite{BC} and Baker \cite{Bak2} on integral points on elliptic curves and hyperelliptic curves, respectively.  Quantitative results in this context were proven by Bilu in \cite{Bilu3}.  We refer the reader to \cite{Bilu} for more general statements and other cases of an effective Siegel's theorem.  

When $C$ is a curve of genus $g\geq 2$, Siegel's theorem is superseded by Faltings' theorem (Mordell's Conjecture), which states that in this case the set of rational points $C(k)$ is finite for any number field $k$.  Faltings' proof of the Mordell conjecture \cite{Fal} used a reduction (due to Parshin) to the Shafarevich conjecture, proved by Faltings in the same paper.  Moreover, R\'emond has shown \cite{Rem} that an effective version of the Shafarevich conjecture would imply an effective version of Faltings' theorem.  In a similar vein, we will show that Problem \ref{mp} can be reduced to proving an effective version of a restricted form of the Shafarevich conjecture (hyperelliptic Jacobians).  For instance, we will show that an effective version of Siegel's theorem for genus two curves follows from an effective version of the Shafarevich conjecture for abelian surfaces.  

Before stating the main theorem, we make a few more definitions.  For a nonsingular projective curve $C$, we let $\Jac(C)$ denote the Jacobian of $C$.  Let $g\geq 2$ be an integer, $k$ a number field, and $S$ a finite set of places of $k$, which we will always assume contains the archimedean places.  Define
\begin{multline*}
\mathcal{H}(g,k,S)=\{\text{$k$-isomorphism classes of (nonsingular projective) hyperelliptic }\\
\text{curves over $k$ of genus $g$ with good reduction outside $S$}\}
\end{multline*}
and
\begin{multline*}
\mathcal{H}'(g,k,S)=\{\text{$k$-isomorphism classes of hyperelliptic curves $C$ over $k$ of genus $g$}\\
\text{  such that $\Jac(C)$ has good reduction outside $S$}\}.
\end{multline*}

As is well known, if $C$ has good reduction at a place $v$ then so does $\Jac(C)$.  In other words, we have $\mathcal{H}(g,k,S)\subset \mathcal{H}'(g,k,S)$.  In general, this set inclusion is proper.  In fact, as we will see, understanding the difference between the two sets is in some sense the key to solving Problem \ref{mp}.  It follows from the Shafarevich conjecture for abelian varieties, proved by Faltings, that the set $\mathcal{H}'(g,k,S)$ is finite.  We now state our main result.

\begin{theorem}
\label{mt}
Let $g\geq 2$ be an integer.  Suppose that for any number field $k$ and any finite set of places $S$ of $k$ the set $\mathcal{H}'(g,k,S)$ is effectively computable (e.g., an explicit hyperelliptic Weierstrass equation for each element of the set is given).  Then for any number field $k$, any finite set of places $S$ of $k$, any hyperelliptic curve $C$ over $k$ of genus $g$, and any rational function $f\in k(C)$, the set of $S$-integral points with respect to $f$,
\begin{equation*}
C(f,k,S)=\{P\in C(k)\mid f(P)\in \O_{k,S}\},
\end{equation*}
is effectively computable.
\end{theorem}

We now give a brief outline of the proof of Theorem \ref{mt}.  To each point $P\in C(f,k,S)$ we associate a certain double cover $\pi_P:\tilde{C}_P\to C$ such that $\tilde{C}_P$ has good reduction outside some fixed finite set of places $T$ of some number field $L$ (with $T$ and $L$ independent of $P$).  Associated to $\pi_P$ we have a classical construction, the Prym variety $\Prym(\tilde{C}_P/C)$.  When $C$ is hyperelliptic, it will turn out that $\Prym(\tilde{C}_P/C)$ is isomorphic to the Jacobian $\Jac(X_P)$ of some hyperelliptic curve $X_P$ and $\Jac(X_P)$ has good reduction outside of $T$.  This yields a map $C(f,k,S)\to \mathcal{H}'(g,L,T)$, $P\mapsto X_P$, and if $\mathcal{H}'(g,L,T)$ is known then $C(f,k,S)$ can be explicitly computed.  We note that an analogous construction essentially works for any curve $C$ (except that $\Prym(\tilde{C}_P/C)$ will not be a hyperelliptic Jacobian, or even a Jacobian, for general $C$), but we focus on the hyperelliptic case as there are reasons to believe that effectively computing $\mathcal{H}'(g,k,S)$ may be a tractable problem.

Indeed, a primary reason for this hope is that the set $\mathcal{H}(g,k,S)$ is known to be effectively computable, i.e., there is an effective version of the Shafarevich conjecture for hyperelliptic curves.
\begin{theorem}[Effective Shafarevich conjecture for hyperelliptic curves, von K\"anel \cite{vK}]
\label{tEG}
Let $g\geq 2$ be an integer, $k$ a number field, and $S$ a finite set of places of $k$.  The set $\mathcal{H}(g,k,S)$ is effectively computable.  \end{theorem}

The finiteness of $\mathcal{H}(g,k,S)$ goes back to work of Shafarevich \cite{Sha}, Merriman \cite{Mer}, Parshin \cite{Par}, and Oort \cite{Oort}.  Building on earlier work of Merriman and Smart \cite{MS}, Smart \cite{Smart} explicitly computed the set $\mathcal{H}(2,\mathbb{Q},\{2,\infty\})$ using effective results of Evertse and Gy\"ory \cite{EG} on related problems concerning discriminants of binary forms.  Using the results of Evertse and Gy\"ory \cite{EG} and a result of Liu \cite{Liu}, von K\"anel \cite{vK} proved explicit bounds for the heights of Weierstrass models of the hyperelliptic curves represented in $\mathcal{H}(g,k,S)$.  We note that the much earlier proof of Oort \cite{Oort} also yields a certain weaker version of Theorem \ref{tEG}.\footnote{Oort's proof of finiteness in \cite{Oort} (implicitly) yields a fixed computable finite extension $L$ of $k$ and effective bounds for the heights of Weierstrass models {\it over $L$} of the hyperelliptic curves represented in $\mathcal{H}(g,k,S)$.}

In \cite{Poo}, Poonen raised the problem of extending the computations of Merriman and Smart to an effective computation of $\mathcal{H}'(2,\mathbb{Q},\{2,\infty\})$.  Our results give additional motivation and importance, coming from Siegel's theorem, to the problem of extending Theorem \ref{tEG} from the set $\mathcal{H}(g,k,S)$ to the set $\mathcal{H}'(g,k,S)$.  In view of Theorems \ref{mt} and \ref{tEG}, to solve Problem \ref{mp} it would suffice to solve the following problem:

\begin{problem}
Find an effective bound $B(g,k,S)$ such that if $C$ is a hyperelliptic curve of genus $g$ over $k$ and $\Jac(C)$ has good reduction outside $S$, then $C$ has good reduction at all primes $\mathfrak{p}$ of $k$ with norm $\Norm(\mathfrak{p})> B(g,k,S)$.
\end{problem}

Note that an ineffective bound $B(g,k,S)$ follows trivially from the finiteness of $\mathcal{H}'(g,k,S)$.

\section{Parshin's construction}

Parshin \cite{Par} was the first to notice that the Mordell conjecture could be obtained as a consequence of the Shafarevich conjecture.  The basic idea is to associate to each rational point $P\in C(k)$ a covering $\pi_P:\tilde{C}_P\to C$ such that the curve $\tilde{C}_P$ has good reduction outside a finite set $S$ (independent of $P$) and $\pi_P$ has certain specified ramification.  We will use a similar construction to study integral points on curves.

\begin{theorem}
\label{Par}
Let $C$ be a nonsingular projective curve over a number field $k$, $f\in k(C)$ a nonconstant rational function, and $S$ a finite set of places of $k$ containing the archimedean places.  Fix a pole $Q\in C(\kbar)$ of $f$.  There exists a number field $L\supset k$ and a finite set of places $T$ of $L$ with the following property:   for any $P\in C(f,k,S)$ and any double cover $\pi:\tilde{C}\to C$ over $\kbar$ ramified exactly above $P$ and $Q$, there exists a morphism of nonsingular projective curves $\pi':\tilde{C}'\to C$ with the following properties:
\begin{enumerate}
\item  There is an isomorphism $\psi:\tilde{C}\to \tilde{C}'$ over $\kbar$ such that $\pi=\pi'\circ\psi$.
\item  $\pi'$ and $\tilde{C}'$ are both defined over $L$.
\item  $\deg \pi'=2$.
\item  $\pi'$ is ramified exactly above $P$ and $Q$.
\item  Both $\tilde{C}'$ and $C$ have good reduction outside $T$.
\end{enumerate}
\end{theorem}

We note that for any $P$ and $Q$ such maps $\pi$ exist (see Theorem \ref{bij}).  We will use the following result proven by Silverman in \cite[Prop.\ 3, Lemma 4]{Sil}, where it is attributed to Szpiro and Ogus.

\begin{lemma}
\label{OS}
Let $C$ be a nonsingular projective curve over a number field $L$ and let $g\in L(C)$.  Suppose that
\begin{equation*}
\dv(g)=\pm P_1 \pm P_2\pm \cdots \pm P_n +pD,
\end{equation*}
where $P_1,\ldots, P_n\in C(\Lbar)$, $D$ is some divisor on $C$, and $p$ is a prime.  Let $\mathscr{C}$ be a model for $C$ over $\O_L$.  Let $T$ be a finite set of places of $L$ such that:
\begin{enumerate}
\item  $T$ contains all archimedean places of $L$.
\label{l1}
\item  $T$ contains all places of bad reduction of $\mathscr{C}$.
\item  $T$ contains all places of $L$ lying above $p$.
\item  The ring of $T$-integers $\O_{L,T}$ has class number one.
\label{l4}
\end{enumerate}
There exists an element $\alpha\in L^*$ such that if $C'$ is a nonsingular projective curve with function field $L(C)(\sqrt[p]{\alpha g})$ and $v\not\in T$ is a place of $L$ such that $P_1,\ldots, P_n$ are distinct modulo $v$, then $C'$ has good reduction at $v$.
\end{lemma}

We now prove Theorem \ref{Par}.

\begin{proof}[Proof of Theorem \ref{Par}]
Without loss of generality, we may assume that $Q\in C(k)$ and that all of the $2$-torsion of $\Jac(C)$ is rational over $k$.  Let $L$ be a number field such that
\begin{equation}
\label{MW}
\Jac(C)(k)\subset 2\Jac(C)(L).
\end{equation}
The existence of such an $L$ follows from the (weak) Mordell-Weil theorem and we recall that such a field $L$ can be explicitly given.  We may assume without loss of generality that $S$ is a finite set of places of $k$ containing all places of bad reduction of $C$, all places of $k$ lying above the prime $2$, and such that $\O_{k,S}$ has trivial class group.  Let $s_1,\ldots, s_n$ be generators of $\O_{k,S}^*$.  Then $L=k(\sqrt{s_1},\ldots,\sqrt{s_n})$ will be a number field satisfying \eqref{MW} (\cite[\S C.1]{HS}).

Let $\mathscr{C}$ be a model of $C$ over $\O_L$.  Let $T$ be a finite set of places of $L$ satisfying \eqref{l1}--\eqref{l4} of Lemma \ref{OS} (with $p=2$) and such that

\begin{enumerate}
\setcounter{enumi}{4}
\item  $T$ contains every place of $L$ lying above a place of $S$.
\item  $T$ contains every finite place $v$ of $L$ such that $f$ is identically $0$ or $\infty$ modulo $v$.
\item  $T$ contains every finite place $v$ of $L$ such that a zero of $f$ reduces to a pole of $f$ modulo $v$.
\end{enumerate}
Let $P\in C(f,k,S)$.  Let $\pi:\tilde{C}\to C$ be a double cover (over $\kbar$) ramified exactly above $P$ and $Q$.  Then $\pi$ corresponds to an extension of function fields $\kbar(C)\subset \kbar(C)(\sqrt{g})$ for some rational function $g\in \kbar(C)$.  It is a standard fact that $\pi$ is ramified above a point $R\in C(\kbar)$ if and only if $g$ has a pole or zero of odd order at $R$.  Since $\pi$ is ramified exactly above $P$ and $Q$, we must have
\begin{equation*}
\dv(g)=P-Q+2D
\end{equation*}
for some divisor $D$.  Since $\Jac(C)(k)\subset 2\Jac(C)(L)$, the divisor class $[D]$ is $L$-rational and $D\sim E$ for some $L$-rational divisor $E$.  Then for an appropriate rational function $h\in \kbar(C)$ satisfying $\dv(h)=E-D$, after replacing $g$ by $gh^2$ we can assume that $g\in L(C)$.

Let $v\not\in T$ be a place of $L$.  Since $f(P)\in \O_{k,S}\subset \O_{L,T}$ by assumption, it follows from the definition of $T$ that $P$ cannot reduce to the pole $Q$ of $f$ modulo $v$.  Then by Lemma \ref{OS}, there exists $\alpha\in L^*$ such that if $\tilde{C}'$ is the nonsingular projective curve with function field $L(C)(\sqrt{\alpha g})$, then $\tilde{C}'$ has good reduction outside $T$.  Let $\pi':\tilde{C}'\to C$ be the morphism corresponding to the inclusion of function fields $L(C)\subset L(C)(\sqrt{\alpha g})$.  Then $\tilde{C}'$ and $\pi'$ satisfy the required properties.
\end{proof}

\section{Prym Varieties}

Let $\pi:\tilde{C}\to C$ be a morphism of nonsingular projective curves of degree two.  There are two natural maps between the Jacobians $\Jac(\tilde{C})$ and $\Jac(C)$ that one can associate to $\pi$.  First, we have the pullback map 
\begin{align*}
\pi^*:\Jac(C)&\to \Jac(\tilde{C}),\\
[D]&\mapsto [\pi^*D],
\end{align*}
where $[D]$ denotes the divisor class of a divisor $D$ of degree $0$.  Second, we have the so-called norm map
\begin{align*}
\Nm:\Jac(\tilde{C})&\to \Jac(C),\\
\left[\sum n_PP\right]&\mapsto \left[\sum n_P\pi(P)\right].
\end{align*}

\begin{definition}
The Prym variety associated to the double cover $\pi:\tilde{C}\to C$ is defined by
\begin{equation*}
\Prym(\tilde{C}/C)=(\ker \Nm)^0,
\end{equation*}
the connected component of $\ker \Nm$ containing the identity.
\end{definition}
We recall some basic facts about Prym varieties (see \cite{Mum}).  Let $\iota:\tilde{C}\to \tilde{C}$ be the involution of $\tilde{C}$ interchanging the two sheets of $\pi$.  This induces an involution $\iota:\Jac(\tilde{C})\to\Jac(\tilde{C})$.  Then we have the identities
\begin{equation*}
\Prym(\tilde{C}/C)=\ker(1+\iota)^0=\im (1-\iota).
\end{equation*}
Let $i:\Prym(\tilde{C}/C)\to \Jac(\tilde{C})$ be the inclusion.  Assume further now that $C$ has positive genus and that $\pi$ is either \'etale or ramified above exactly two points of $C$.  We have an isogeny 
\begin{align*}
\Jac(C)\times \Prym(\tilde{C}/C)&\to \Jac(\tilde{C}),\\
(P,Q)&\mapsto \pi^*(P)+i(Q).
\end{align*}
Furthermore, if $\theta_{\tilde{C}}$ is a theta divisor on $\Jac(\tilde{C})$, then $i^*\theta_{\tilde{C}}\equiv 2\Xi$ for some ample divisor $\Xi$ on $\Prym(\tilde{C}/C)$ and $\Xi$ yields a principal polarization of $\Prym(\tilde{C}/C)$.  So $\Prym(\tilde{C}/C)$ can naturally be given the structure of a principally polarized abelian variety.  We will find the following lemma useful.

\begin{lemma}
\label{Serre}
Suppose that $\tilde{C}$, $C$, and $\pi:\tilde{C}\to C$ are defined over a number field $k$.  Let $v$ be a finite place of $k$.  Then $\Jac(\tilde{C})$ has good reduction at $v$ if and only if both $\Prym(\tilde{C}/C)$ and $\Jac(C)$ have good reduction at $v$.
\end{lemma}
\begin{proof}
The result follows \cite[Cor.\ 2]{Ser} from the fact that the abelian varieties $\Jac(\tilde{C})$ and $\Jac(C)\times \Prym(\tilde{C}/C)$ are $k$-isogenous.
\end{proof}

In general, the Prym variety $\Prym(\tilde{C}/C)$ is not the Jacobian of a curve.  However, in the important special case where $C$ is hyperelliptic this holds (and the Prym variety is a hyperelliptic Jacobian).

\begin{theorem}[Dalaljan \cite{Dal1,Dal2}, Mumford \cite{Mum}]
\label{Dal}
Suppose that $C$ is hyperelliptic and $\pi:\tilde{C}\to C$ is a double cover of $C$ that is either unramified or ramified above exactly two points of $C$.  Then $(\Prym(\tilde{C}/C),\Xi)$ is isomorphic to the Jacobian $(\Jac(C'),\theta_{C'})$ of some hyperelliptic curve $C'$.
\end{theorem}

Suppose that $C$ is a hyperelliptic curve of genus $g$.  For $Q\in C(\kbar)$ and $P\in C(f,k,S)$ as in Theorem \ref{Par}, we have a double cover $\tilde{C}_P\to C$ ramified exactly above $P$ and $Q$ and having nice reduction properties.  By Theorem \ref{Dal}, $\Prym(\tilde{C}_P/C)\cong \Jac(X_P)$ for some hyperelliptic curve $X_P$ of genus $g$.  In the next section we will explicitly compute a Weierstrass equation for $X_P$.

\section{Some explicit equations}

We begin by recalling some results from \cite[\S 3]{Dal2}.

\begin{theorem}[Dalaljan]
\label{tDal}
Let $k$ be an algebraically closed field, $\chr k\neq 2$.  Let $C$ be a nonsingular projective hyperelliptic curve over $k$ with hyperelliptic map $p:C\to \mathbb{P}^1$.  Let $\pi:\tilde{C}\to C$ be a double cover of nonsingular projective curves, ramified above exactly two points of $C$.  Associated to $\pi$ and $p$ there exists a tower of curves
\begin{equation}
\label{tower}
\begindc{\commdiag}[57]
\obj(2,0)[20]{$\mathbb{P}^1$}
\obj(1,1)[11]{$C$}
\obj(2,1)[21]{$C_0$}
\obj(3,1)[31]{$C_1$}
\obj(0,2)[02]{$\tilde{C}'$}
\obj(1,2)[12]{$\tilde{C}$}
\obj(2,2)[22]{$\tilde{C}_0$}
\obj(3,2)[32]{$\tilde{C}_1$}
\obj(4,2)[42]{$\tilde{C}_1'$}
\obj(2,3)[23]{$\tilde{\tilde{C}}$}

\mor{11}{20}{$p$}[\atright, \solidarrow]
\mor{21}{20}{$p_0$}
\mor{31}{20}{$p_1$}

\mor{02}{11}{$\pi'$}[\atright, \solidarrow]
\mor{12}{11}{$\pi$}[\atright, \solidarrow]
\mor{22}{11}{$\pi_0$}[\atright, \solidarrow]
\mor{22}{21}{$\pi_0''$}
\mor{22}{31}{$\pi_0'$}
\mor{32}{31}{$\pi_1$}
\mor{42}{31}{$\pi_1'$}

\mor{23}{02}{$\tilde{\pi}'$}[\atright, \solidarrow]
\mor{23}{12}{$\tilde{\pi}$}[\atright, \solidarrow]
\mor{23}{22}{$\tilde{\pi}_0$}
\mor{23}{32}{$\tilde{\pi}_1$}
\mor{23}{42}{$\tilde{\pi}_1'$}
\enddc
\end{equation}
with $C_1\cong \mathbb{P}^1$, $\tilde{C}_1$ hyperelliptic, $\deg \pi_1=2$, and
\begin{equation*}
(\Prym(\tilde{C}/C),\Xi)\cong (\Jac(\tilde{C}_1),\theta_{\tilde{C}_1}).
\end{equation*}
For any positive integer $g$, the tower \eqref{tower} induces a bijection between equivalence classes of towers
\begin{multline*}
\{[\tilde{C}\stackrel{\pi}{\rightarrow}C\stackrel{p}{\rightarrow}\mathbb{P}^1]\mid \deg \pi=\deg p=2, g(\tilde{C})=2g, g(C)=g \}\longleftrightarrow\\
 \{[\tilde{C}_1\stackrel{\pi_1}{\rightarrow}C_1\stackrel{p_1}{\rightarrow}\mathbb{P}^1]\mid \deg \pi_1=\deg p_1=2, g(\tilde{C}_1)=g, C_1\cong \mathbb{P}^1\}.
\end{multline*}
\end{theorem}

Here, we say that two towers of curves $X\stackrel{\pi}{\rightarrow} Y\stackrel{p}{\rightarrow} Z$ and $X'\stackrel{\pi'}{\rightarrow} Y'\stackrel{p'}{\rightarrow} Z'$ are equivalent if there exists a commutative diagram
\begin{equation*}
\begindc{\commdiag}[40]
\obj(0,2)[02]{$X$}
\obj(0,1)[01]{$Y$}
\obj(0,0)[00]{$Z$}
\obj(1,2)[12]{$X'$}
\obj(1,1)[11]{$Y'$}
\obj(1,0)[10]{$Z'$}

\mor{02}{01}{$\pi$}
\mor{01}{00}{$p$}
\mor{12}{11}{$\pi'$}
\mor{11}{10}{$p'$}
\mor{02}{12}{}
\mor{01}{11}{}
\mor{00}{10}{}
\enddc
\end{equation*}
where the horizontal morphisms are isomorphisms.

\begin{theorem}
\label{bij}
Let $k$ be an algebraically closed field, $\chr k\neq 2$.  Let $C$ be a (nonsingular projective) hyperelliptic curve over $k$ of genus $g$ given by an equation
\begin{equation*}
y^2=\prod_{i=1}^{2g+1}(x-\alpha_i), \quad \alpha_i\in k,
\end{equation*}
and $p:C\to \mathbb{P}^1$ the hyperelliptic map $(x,y)\mapsto x$.  Let $P=(x_P,y_P), Q=(x_Q,y_Q)\in C(k)$ with $x_P\neq x_Q$ and $y_P\neq 0$.  The tower \eqref{tower} induces a bijection between the sets
\begin{equation*}
\{\text{double covers } \pi:\tilde{C}\to C, \text{ up to equivalence, ramified exactly above $P$ and $Q$}\}
\end{equation*}
and
\begin{equation*}
\left\{\text{hyperelliptic curves }\tilde{C}_1: y^2=(x-1)\prod_{i=1}^{2g+1}\left(x-\beta_i\right)\mid \beta_i^2=\frac{x_Q-\alpha_i}{x_P-\alpha_i}, \prod_{i=1}^{2g+1} \beta_i=\frac{y_Q}{y_P}\right\}.
\end{equation*}

Both sets have cardinality $2^{2g}$.
\end{theorem}

\begin{proof}

We first give an explicit description of the set of double covers $\pi:\tilde{C}\to C$ ramified exactly above $P$ and $Q$.  At the function field level, we have $k(\tilde{C})=k(C)(\sqrt{\phi})$ for some rational function $\phi\in k(C)$.  The corresponding morphism $\pi$ is ramified exactly above $P$ and $Q$ if and only if $\dv(\phi)=P+Q+2D$ for some divisor $D$.  It's easily seen then that there is a bijection
\begin{multline*}
\{\text{double covers } \pi:\tilde{C}\to C, \text{ up to equivalence, ramified exactly above $P$ and $Q$}\} \\
\longleftrightarrow\{[D]\mid D\in \Div(C), 2D\sim -(P+Q)\},
\end{multline*}
where $\Div(C)$ denotes the group of divisors of $C$ and $[D]$ the linear equivalence class of $D$.  As is well-known, the latter (and hence former) set has cardinality $2^{2g}$.  

Let $\infty$ denote the (unique) point of $C$ at infinity relative to the equation $y^2=\prod_{i=1}^{2g+1}(x-\alpha_i)$.  Then we have an embedding $C\to \Jac(C)$, $\pt\mapsto [\pt-\infty]$.  Every element of $\Jac(C)$ can be represented by a divisor of the form $P_1+\cdots+P_g-g\infty$, $P_1,\ldots,P_g\in C(k)$.  Then in each linear equivalence class of divisors $D$ satisfying $2D\sim -(P+Q)$, we can find a divisor $D$ of the form
\begin{equation*}
D=P_1+\cdots+P_g-(g+1)\infty, \quad P_1,\ldots,P_g\in C(k).
\end{equation*}
Let $\pi:\tilde{C}\to C$ be a double cover ramified exactly above $P$ and $Q$.  Then we find that $k(\tilde{C})=k(C)(\sqrt{\phi})$ for some rational function $\phi\in k(C)$ satisfying
\begin{equation*}
\dv(\phi)=P+Q+2P_1+\cdots+2P_g-(2g+2)\infty.
\end{equation*}
As usual, let $L(D)$ be the $k$-vector space $L(D)=\{\psi \in k(X)\mid\dv(\psi)\geq -D\}$.  By Riemann-Roch, $\dim L((2g+2)\infty)=g+3$ and a basis for $L((2g+2)\infty)$ is given by the rational functions $1,x,\ldots,x^{g+1},y$.  It follows that we can write
\begin{equation*}
\phi=ay+h(x)
\end{equation*}
for some $a\in k$ and some $h\in k[x]$ with $\deg h\leq g+1$.  Since $x_P\neq x_Q$, we have $a\neq 0$, and so we can assume $a=1$.  Let $f(x)=\prod_{i=1}^{2g+1}(x-\alpha_i)$.  Then $\tilde{C}$ has an equation of the form
\begin{align*}
y^2&=f(x),\\
z^2&=y+h(x),
\end{align*}
where $h\in k[x]$ and $\deg h\leq g+1$ (where throughout, we mean that the curve, here $\tilde{C}$, is a projective normalization of the curve defined by the given equations).

We now describe the relevant part of the tower \eqref{tower} (see \cite{Dal2}).  For a double cover $\phi:X'\to X$ of nonsingular projective curves, we let $i_\phi$ denote the corresponding involution of $X'$.  We set $\tilde{C}'=\tilde{C}$ and $\pi'=i_p\circ \pi$, where $p:C\to \mathbb{P}^1$ is the projection onto the $x$-coordinate.  Let $\tilde{\tilde{C}}=\tilde{C}\times_{C}\times \tilde{C}'$ and let $\tilde{\pi}$ and $\tilde{\pi}'$ be the natural projection maps onto $\tilde{C}$ and $\tilde{C}'$, respectively.  Explicitly, $\tilde{\tilde{C}}$ consists of pairs $(\tilde{P},\tilde{P}')\in \tilde{C}\times \tilde{C}'$ such that $\pi(\tilde{P})=\pi'(\tilde{P}')$.  Then $\tilde{\tilde{C}}$ can be given by the equations
\begin{align*}
y^2&=f(x),\\
z^2&=y+h(x),\\
z'^2&=-y+h(x).
\end{align*}

Let $i_{\tilde{\pi}_0}=i_{\tilde{\pi}'}\circ i_{\tilde{\pi}}$ and let $i_{\tilde{\pi}_1}$ be the involution of $\tilde{\tilde{C}}=\tilde{C}\times_{C}\times \tilde{C}'$ that switches the coordinates.  In the coordinates $(x,y,z,z')$ these are the involutions $(x,y,z,z')\mapsto (x,y,-z,-z')$ and $(x,y,z,z')\mapsto (x,-y,z',z)$, respectively.  Taking the quotients of $\tilde{\tilde{C}}$ by $i_{\tilde{\pi}_0}$ and $i_{\tilde{\pi}_1}$, we obtain curves $\tilde{C}_0$ and $\tilde{C}_1$, respectively, and double covers $\tilde{\pi}_0:\tilde{\tilde{C}}\to \tilde{C}_0$ and $\tilde{\pi}_1:\tilde{\tilde{C}}\to \tilde{C}_1$, respectively.  Note that
\begin{align*}
(z+z')^2&=2h(x)+2zz',\\
(zz')^2&=h(x)^2-f(x).
\end{align*}

Using the above equations, we see that $\tilde{C}_0$ is given by the equations
\begin{align*}
y^2&=f(x),\\
z^2&=h(x)^2-f(x),
\end{align*}
and $\tilde{C}_1$ by the equations
\begin{align}
w^2&=2h(x)+2z,\label{C11}\\
z^2&=h(x)^2-f(x)\label{C12}.
\end{align}

The involutions $i_{\tilde{\pi}_0}$ and $i_{\tilde{\pi}_1}$ commute and generate a group of order $4$.  The quotient of $\tilde{\tilde{C}}$ modulo the action of this group yields the curve $C_1$ in the tower.  The curve $C_1$ is naturally given as the projective normalization of the curve defined by the equation
\begin{equation*}
z^2=h(x)^2-f(x).
\end{equation*}

The induced maps $\pi_0'$ and $\pi_1$ are then induced by the natural projection maps onto the $x$ and $z$ coordinates.  From the definition of $h(x)$, we have
\begin{equation}
\label{C12p}
h(x)^2-f(x)=(x-x_P)(x-x_Q)F(x)^2,
\end{equation}
for some polynomial $F\in k[x]$.  It follows that $C_1\cong \mathbb{P}^1$.  As $\deg \pi_0'=\deg \pi_1=2$, we see that $\tilde{C}_0$ and $\tilde{C}_1$ are hyperelliptic curves.  

A hyperelliptic Weierstrass equation for $\tilde{C}_1$ can be computed as follows.  In view of \eqref{C12p}, we can parametrize \eqref{C12} by setting 
\begin{align}
t^2&=\frac{x-x_Q}{x-x_P},\label{t2}\\
x&=x(t)=\frac{x_Q-x_Pt^2}{1-t^2},\notag\\
z&=z(t)=t(x(t)-x_P)F(x(t)).\notag
\end{align}
Substituting into \eqref{C11}, we see that we need to consider the polynomial
\begin{align*}
G(t)&=(1-t^2)^{g+1}(h(x(t))+z(t))\\
&=(1-t^2)^{g+1}\left(h\left(\frac{x_Q-x_Pt^2}{1-t^2}\right)+t\left(\frac{x_Q-x_P}{1-t^2}\right)F\left(\frac{x_Q-x_Pt^2}{1-t^2}\right)\right).
\end{align*}

Let $\alpha_i, i=1,\ldots, 2g+1$, be the roots of $f$.  Then when $x=\alpha_i$, \eqref{C12} gives $z=\pm h(\alpha_i)$.  From \eqref{t2}, it then follows that for some (unique) choice of the square root, $t=\sqrt{\frac{x_Q-\alpha_i}{x_P-\alpha_i}}=\beta_i$ is a root of $G(t)$, $i=1,\ldots, 2g+1$.   From \eqref{C12p}, it follows that the leading coefficients of $h(x)$ and $F(x)$ are equal up to sign.  This easily implies that either $t=1$ or $t=-1$ is a root of $G(t)$.  Replacing $F(x)$ by $-F(x)$, if necessary, we can assume that $t=1$ is a root of $G(t)$.  Since $\deg G\leq 2g+2$, we find that the roots of $G$ are exactly given by the $2g+2$ distinct elements $t=1,\beta_1,\ldots, \beta_{2g+1}$.  

We have $G(0)=h(x_Q)=-y_Q$, from the definition of $h$.  The leading coefficient of $G$ can be computed as
\begin{equation*}
\lim_{t\to\infty}(-1)^{g+1}\frac{G(t)}{(1-t^2)^{g+1}}=(-1)^{g+1}h(x_P)=-(-1)^{g+1}y_P.
\end{equation*}
Thus,
\begin{equation*}
1\cdot \prod_{i=1}^{2g+1} \beta_i=\frac{-y_Q}{-(-1)^{g+1}y_P}=(-1)^{g+1}\frac{y_Q}{y_P}.
\end{equation*}
If $g$ is odd, so that $(1-t^2)^{g+1}$ is an even power of $1-t^2$, then it follows from the above and \eqref{C11} and \eqref{C12} that $\tilde{C}_1$ has a Weierstrass equation
\begin{equation}
\label{finaleq}
y^2=(x-1)\prod_{i=1}^{2g+1}\left(x-\beta_i\right), \qquad \beta_i^2=\frac{x_Q-\alpha_i}{x_P-\alpha_i}, \prod_{i=1}^{2g+1} \beta_i=\frac{y_Q}{y_P}.
\end{equation}
If $g$ is even, then
\begin{equation*}
\frac{(1-t^2)^{g+2}}{(t-1)^2}(h(x(t))+z(t))=c'(t+1)\prod_{i=1}^{2g+1}\left(t-\beta_i\right),
\end{equation*}
for some constant $c'$ and where $\prod_{i=1}^{2g+1} \beta_i=-\frac{y_Q}{y_P}$.  Replacing $t$ by $-t$, we again see that $\tilde{C}_1$ has a Weierstrass equation \eqref{finaleq}.

In the coordinates of \eqref{finaleq}, the tower $\tilde{C}_1\stackrel{\pi_1}{\rightarrow} C_1\stackrel{p_1}{\rightarrow}\mathbb{P}^1$ is given by the hyperelliptic map $(x,y)\mapsto x$ followed by the squaring map $C_1=\mathbb{P}^1\to \mathbb{P}^1$, $x\mapsto x^2$.  The tower $\tilde{C}_1\stackrel{\pi_1}{\rightarrow} C_1\stackrel{p_1}{\rightarrow}\mathbb{P}^1$ is associated to the tower $\tilde{C}\stackrel{\pi}{\rightarrow} C\stackrel{p}{\rightarrow}\mathbb{P}^1$, and from the above, $\tilde{C}_1$ has an equation given by \eqref{finaleq}.  Note that there are exactly $2^{2g}$ possibilities for the curve \eqref{finaleq}, and hence for the tower $[\tilde{C}_1\stackrel{\pi_1}{\rightarrow} C_1\stackrel{p_1}{\rightarrow}\mathbb{P}^1]$, given by the $2^{2g}$ possibilities for $\beta_i$ subject to the constraint $\prod_{i=1}^{2g+1} \beta_i=\frac{y_Q}{y_P}$.  Since the map
\begin{equation*}
[\tilde{C}\stackrel{\pi}{\rightarrow} C\stackrel{p}{\rightarrow}\mathbb{P}^1]\mapsto [\tilde{C}_1\stackrel{\pi_1}{\rightarrow} C_1\stackrel{p_1}{\rightarrow}\mathbb{P}^1]
\end{equation*}
given by the tower \eqref{tower} is injective and there are $2^{2g}$ possibilities for $[\tilde{C}\to C\to \mathbb{P}^1]$, it follows that all of the $2^{2g}$ possibilities for $\tilde{C}_1$ arise from some covering $\tilde{C}\to C$, proving the theorem.
\end{proof}

\begin{corollary}
\label{cormain}
Let $C$ be a hyperelliptic curve over a number field $k$ of genus $g$ given by an equation
\begin{equation*}
y^2=\prod_{i=1}^{2g+1}(x-\alpha_i), \quad \alpha_i\in k.
\end{equation*}
Let $f\in k(C)$ be a nonconstant rational function and $S$ a finite set of places of $k$ containing the archimedean places.  Fix a pole $Q=(x_Q,y_Q)\in C(\kbar)$ of $f$.  There exists a number field $L\supset k$ and a finite set of places $T$ of $L$ such that for any $P=(x_P,y_P)\in C(f,k,S)$ with $x_P\neq x_Q$, $y_P\neq 0$, any $\beta_i$ satisfying $\beta_i^2=\frac{x_Q-\alpha_i}{x_P-\alpha_i}, \prod_{i=1}^{2g+1} \beta_i=\frac{y_Q}{y_P}$, and some $c=c(\beta_1,\ldots,\beta_{2g+1})\in L^*$, the Jacobian $\Jac(X_P)$ of the curve
\begin{equation*}
X_P: y^2=c(x-1)\prod_{i=1}^{2g+1}\left(x-\beta_i\right)
\end{equation*}
is defined over $L$ and has good reduction outside of $T$.
\end{corollary}

\begin{proof}
Without loss of generality, we may assume that $Q\in C(k)$.  Let $L$ and $T$ be as in Theorem \ref{Par}.  Let $P=(x_P,y_P)\in C(f,k,S)$ with $x_P\neq x_Q$, $y_P\neq 0$.  Let $\beta_i$, $i=1,\ldots, 2g+1$, satisfy $\beta_i^2=\frac{x_Q-\alpha_i}{x_P-\alpha_i}, \prod_{i=1}^{2g+1} \beta_i=\frac{y_Q}{y_P}$.  If $P'=(x',y')\in C(k)$, we have $x'-\alpha_i=a_iz_i^2$ for some $z_i\in k$ and some $a_i\in k$ divisible only by primes of $k$ dividing the discriminant of $\prod_{i=1}^{2g+1}(x-\alpha_i)$.  Then there is a fixed finite extension of $L$ (independent of $P'\in C(k)$) containing $\sqrt{x'-\alpha_i}=z_i\sqrt{a_i}$ for all $i$.  By replacing $L$ by this finite extension (and replacing $T$ by the set of places lying above it), we can assume that we always have $\beta_i\in L$, $i=1,\ldots, 2g+1$.

Let $X_P'$ be the curve $y^2=(x-1)\prod_{i=1}^{2g+1}\left(x-\beta_i\right)$.  By Theorem~\ref{Par} and Theorem~\ref{bij}, there is a double covering $\pi_P:\tilde{C}_P\to C$, ramified exactly above $P$ and $Q$, that corresponds to $\tilde{C}_1=X_P'$ via the tower \eqref{tower}, such that both $\pi_P$ and $\tilde{C}_P$ are defined over $L$ and both $\tilde{C}_P$ and $C$ have good reduction outside $T$.  By Theorem \ref{tDal},
\begin{equation*}
\Prym(\tilde{C}_P/C)\cong \Jac(X_P')
\end{equation*}
over $\Lbar$.  In fact, the proofs of Theorems \ref{tDal} and \ref{bij} show that
\begin{equation*}
\Prym(\tilde{C}_P/C)\cong \Jac(X_P)
\end{equation*}
over $L$, where $X_P$ is some quadratic twist of $X_P'$ given by
\begin{equation*}
X_P: y^2=c(x-1)\prod_{i=1}^{2g+1}\left(x-\beta_i\right),
\end{equation*}
for some $c\in L^*$.  Since both $\tilde{C}_P$ and $C$ have good reduction outside $T$, by Lemma~\ref{Serre}, $\Jac(X_P)$ has good reduction outside $T$.
\end{proof}

An alternative, more direct proof of this result follows from Theorems \ref{binthm} and \ref{gr} in Section \ref{ls}.

\section{Main Theorem}

We now prove the main theorem from the introduction.

\begin{theorem}
\label{mainthm}
Let $g\geq 2$ be an integer.  Suppose that for any number field $k$ and any finite set of places $S$ of $k$ the set $\mathcal{H}'(g,k,S)$ is effectively computable (e.g., an explicit hyperelliptic Weierstrass equation for each element of the set is given).  Then for any number field $k$, any finite set of places $S$ of $k$, any hyperelliptic curve $C$ over $k$ of genus $g$, and any rational function $f\in k(C)$, the set of $S$-integral points with respect to $f$,
\begin{equation*}
C(f,k,S)=\{P\in C(k)\mid f(P)\in \O_{k,S}\},
\end{equation*}
is effectively computable.
\end{theorem}

\begin{proof}
Without loss of generality, by enlarging $k$ we can assume that every Weierstrass point of $C$ is $k$-rational and that some pole $Q=(x_Q,y_Q)$ of $f$ is $k$-rational.  Then $C$ can be given by a hyperelliptic Weierstrass equation
\begin{equation*}
C:y^2=\prod_{i=1}^{2g+1}(x-\alpha_i), \quad \alpha_i\in k.
\end{equation*}
Let $U$ consist of the set of Weierstrass points of $C$ along with $Q$ and its image under the hyperelliptic involution.  By Corollary \ref{cormain}, for some number field $L$ and some finite set of places $T$ of $L$, we have a map (arbitrarily choosing among the choices for $X_P$)
\begin{align*}
C(f,k,S)\setminus U&\to \mathcal{H}'(g,L,T),\\
P&\mapsto X_P.
\end{align*}
Explicitly, we can compute $C(f,k,S)$ from $\mathcal{H}'(g,L,T)$ as follows.  Recall that for $P=(x_P,y_P)\in C(f,k,S)\setminus U$, $X_P$ is defined by an equation
\begin{equation*}
y^2=c_P(x-1)\prod_{i=1}^{2g+1}\left(x-\sqrt{\frac{x_Q-\alpha_i}{x_P-\alpha_i}}\right),
\end{equation*}
for some $c_P\in L^*$ and some appropriate choice of the square roots.  We pick four of the roots of the polynomial on the right-hand side, say $\beta_i=\sqrt{\frac{x_Q-\alpha_i}{x_P-\alpha_i}}, i=1,2,3,4$, and consider the cross-ratio
\begin{equation*}
\CR(\beta_1,\beta_2,\beta_3,\beta_4)=\frac{(\beta_1-\beta_3)(\beta_2-\beta_4)}{(\beta_2-\beta_3)(\beta_1-\beta_4)}.
\end{equation*}
Alternatively, we consider the rational function
\begin{equation}
\label{rat}
\frac{(c_1z_3-c_3z_1)(c_2z_4-c_4z_2)}{(c_2z_3-c_3z_2)(c_1z_4-c_4z_1)}
\end{equation}
on the curve defined by $z_i^2=x-\alpha_i$, $i=1,2,3,4$, where we view $c_i=\sqrt{x_Q-\alpha_i}$, $i=1,2,3,4$, as fixed constants.  Note that by Kummer theory, since the $\alpha_i$ are distinct, we have $[\kbar(x,z_1,\ldots,z_4):\kbar(x)]=2^4=16$.  This immediately implies that the rational function \eqref{rat} is nonconstant and thus that the equation $\CR(\beta_1,\beta_2,\beta_3,\beta_4)=\alpha$ has only finitely many solutions in $x_P$ for any $\alpha\in \kbar$.  Now let $C'\in \mathcal{H}'(g,L,T)$, given by a Weierstrass equation $C':y^2=c'\prod_{i=1}^{2g+2}(x-\gamma_i)$.  If $X_P\cong C'$, then $\CR(\beta_1,\beta_2,\beta_3,\beta_4)=\CR(\gamma_i,\gamma_j,\gamma_k,\gamma_l)$ for some $i,j,k,l\in \{1,\ldots, 2g+2\}$.  Since there are only finitely many possible cross-ratios $\CR(\gamma_i,\gamma_j,\gamma_k,\gamma_l)$, we find that there are only finitely many (explicitly computable) possible points $P=(x_P,y_P)$ with $X_P\cong C'$.  Finally, for each such possible point $P$ and each point $P\in U$, we check if $P\in C(f,k,S)$.
\end{proof}

\section{Binary Forms}
\label{ls}

In this section we give a reformulation of some of our results in terms of binary forms.  Let $k$ be a number field and $S$ a finite set of places of $k$ (which we always assumes contains the archimedean places).  Let $F(X,Z), G(X,Z)\in k[X,Z]$ be binary forms.  Let $U=\left( \begin{array}{cc}
a & b  \\
c & d
\end{array} \right)$, $a,b,c,d\in k$, be a matrix.  Define $F_U(X,Z)=F(aX+bZ,cX+dZ)$.  We will call $F$ and $G$ equivalent if there exists $U\in \GL_2(k)$ and $\lambda\in k^*$ such that $G(X,Z)=\lambda F_U(X,Z)$.  Denote the equivalence class containing $F$ by $[F]$.  Let $\disc(F)$ denote the discriminant of $F$.  Define
\begin{equation*}
\mathcal{B}(r,k,S)=\{[F]\mid F\in \O_{k,S}[X,Z] \text{ is a binary form of degree $r$ and $\disc(F)\in \O_{k,S}^*$}\}.
\end{equation*}

Effective finiteness of the set $\mathcal{B}(r,k,S)$ follows from work of Evertse and Gy\"ory.
\begin{theorem}[Evertse, Gy\"ory \cite{EG}]
Let $r\geq 2$ be an integer, $k$ a number field, and $S$ a finite set of places of $k$.  The set $\mathcal{B}(r,k,S)$ is finite and effectively computable.
\end{theorem}

We now define a larger, but related, set $\mathcal{B'}(r,k,S)\supset \mathcal{B}(r,k,S)$.  The set $\mathcal{B'}(r,k,S)$ contains equivalence classes of certain binary forms $F$ whose discriminant is an $S$-unit outside of primes $\mathfrak{p}$ where $F \pmod{\mathfrak{p}}$ has a factor of multiplicity $\geq 3$.  An effective procedure for computing $\mathcal{B}'(r,k,S)$ would give a solution to Problem \ref{mp} (Corollary \ref{ce}).

More precisely, define $\mathcal{B}'(r,k,S)$ to be the set of equivalence classes of binary forms over $k$ of degree $r$ such that there exists a representative $F\in \O_{k,S}[X,Z]$ satisfying:\\

If $\mathfrak{p}\not\in S$ and $\ord_{\mathfrak{p}}\disc(F)>0$, then $\ord_{\mathfrak{p}}\disc(F)=2mn(n-1)$, where $m$ is some positive integer, $n$ is an odd integer with $3\leq n\leq 2[(r+1)/2]-3$, and $f(x)=F(x,1)$ has $n$ roots $\alpha_1,\ldots,\alpha_n\in k$ with $\ord_{\mathfrak{p}}\alpha_i=2m$, $i=1,\ldots, n$.\\

If $\disc(F)\in \O_{k,S}^*$ or $r\leq 4$, then the above condition is vacuous.  So we trivially have $\mathcal{B}(r,k,S)\subset \mathcal{B}'(r,k,S)$ for all $r$ and $\mathcal{B}(r,k,S)=\mathcal{B}'(r,k,S)$ if $1\leq r\leq 4$.

\begin{theorem}
\label{binthm}
Let $C$ be a hyperelliptic curve over a number field $k$ of genus $g$ given by an equation
\begin{equation*}
y^2=h(x)=\prod_{i=1}^{2g+1}(x-\alpha_i), \quad \alpha_i\in k.
\end{equation*}
Let $f\in k(C)$ be a nonconstant rational function and $S$ a finite set of places of $k$ containing the archimedean places.  Fix a pole $Q=(x_Q,y_Q)\in C(\kbar)$ of $f$.  Let $U$ be the set of Weierstrass points of $C$ along with $Q$ and its image under the hyperelliptic involution. Then there exists a number field $L$ and finite set of places $T$ of $L$ such that we have a well-defined map
\begin{align*}
C(f,k,S)\setminus U&\to \mathcal{B}'(2g+2,L,T),\\
P=(x_P,y_P)&\mapsto \left[(X-Z)\prod_{i=1}^{2g+1}(X-\beta_iZ)\right],
\end{align*}
where the $\beta_i$ are chosen such that $\beta_i^2=\frac{x_Q-\alpha_i}{x_P-\alpha_i}$ and $\prod_{i=1}^{2g+1} \beta_i=\frac{y_Q}{y_P}$.
\end{theorem}

We will see (Corollary \ref{SF}) that $\mathcal{B}'(2g+2,k,S)$ is a finite set.  Then by essentially the same argument as in the proof of Theorem \ref{mainthm}, we obtain the following corollary.

\begin{corollary}
\label{ce}
Let $g\geq 2$ be an integer.  Suppose that for any number field $k$ and any finite set of places $S$ of $k$ the set $\mathcal{B}'(2g+2,k,S)$ is effectively computable.  Then for any number field $k$, any finite set of places $S$ of $k$, any hyperelliptic curve $C$ over $k$ of genus $g$, and any rational function $f\in k(C)$, the set of $S$-integral points with respect to $f$,
\begin{equation*}
C(f,k,S)=\{P\in C(k)\mid f(P)\in \O_{k,S}\},
\end{equation*}
is effectively computable.
\end{corollary}

\begin{proof}[Proof of Theorem \ref{binthm}]
First note that there is an explicit number field $L$, depending on $C$ and $k$, such that for any $P\in C(f,k,S)$ we have $\beta_i\in L$ for all $i$.  Then after enlarging $k$, without of loss of generality it suffices to prove the theorem for points $P\in C(f,k,S)$ such that $\beta_i\in k$, $i=1,\ldots, 2g+1$.  Similarly, by enlarging $S$ we can assume without of loss of generality that $\alpha_i\in \O_{k,S}$, $x_Q-\alpha_i\in \O_{k,S}^*$, $i=1,\ldots, 2g+1$, $\disc(h)\in \O_{k,S}^*$, $\O_{k,S}$ is a principal ideal domain, every place of $k$ lying above $2$ is in $S$, and 
\begin{equation}
\label{meq}
\min\{\ord_\mathfrak{p}(x_P-x_Q),\ord_\mathfrak{p}(y_P-y_Q)\}\leq 0 
\end{equation}
for every $\mathfrak{p}\in M_k\setminus S$ and every $P=(x_P,y_P)\in C(f,k,S)$, where $M_k$ denotes the canonical set of places of $k$ (identifying nonarchimedean places with the corresponding prime ideal).

Let $P\in C(f,k,S)$ with $\beta_i\in k$, $i=1,\ldots, 2g+1$, as in the theorem.  Since $\O_{k,S}$ is principal, we can write $\beta_i=\frac{\gamma_i}{\delta_i}$, where $\gamma_i$ and $\delta_i$ are relatively prime $S$-integers, $i=1,\ldots, 2g+1$.  Consider the binary form
\begin{equation*}
F(X,Z)=(X-Z)\prod_{i=1}^{2g+1}(\delta_iX-\gamma_iZ)\in \O_{k,S}[X,Z],
\end{equation*}
which is equivalent to $(X-Z)\prod_{i=1}^{2g+1}(X-\beta_iZ)$.  Let $\mathfrak{p}\in M_k\setminus S$.  

Suppose that $\ord_{\mathfrak{p}}(x_P-x_Q)=0$.  Then we claim that $\ord_\mathfrak{p}\disc(F)=0$.  For this, it suffices to show that $\ord_\mathfrak{p}(\gamma_i\delta_j-\gamma_j\delta_i)=0$, $i\neq j$, and $\ord_\mathfrak{p}(\gamma_i-\delta_i)=0$ for all $i$.  We have the identity
\begin{equation}
\label{bid}
\beta_i^2-\beta_j^2=\frac{\gamma_i^2\delta_j^2-\gamma_j^2\delta_i^2}{\delta_i^2\delta_j^2}=\frac{x_Q-\alpha_i}{x_P-\alpha_i}-\frac{x_Q-\alpha_j}{x_P-\alpha_j}=\frac{(x_P-x_Q)(\alpha_j-\alpha_i)}{(x_P-\alpha_i)(x_P-\alpha_j)}.
\end{equation}
Since $\disc(h)\in \O_{k,S}^*$, we have $\ord_\mathfrak{p}(\alpha_j-\alpha_i)=0$.  By assumption, $\ord_\mathfrak{p}(x_P-x_Q)=\ord_\mathfrak{p}((x_P-\alpha_i)-(x_Q-\alpha_i))=0$.  This last equality, along with $\ord_\mathfrak{p}(x_Q-\alpha_i)\geq 0$, implies that $\ord_\mathfrak{p}(x_P-\alpha_i)=2\ord_\mathfrak{p}\delta_i$.  Then
\begin{align*}
\ord_\mathfrak{p}(\gamma_i^2\delta_j^2-\gamma_j^2\delta_i^2)&=\ord_\mathfrak{p}((\gamma_i\delta_j-\gamma_j\delta_i)(\gamma_i\delta_j+\gamma_j\delta_i))\\
&=\ord_\mathfrak{p}\frac{\delta_i^2\delta_j^2(x_P-x_Q)(\alpha_j-\alpha_i)}{(x_P-\alpha_i)(x_P-\alpha_j)}=0.
\end{align*}
Since $\gamma_i\delta_j\pm\gamma_j\delta_i$ is an $S$-integer, we find that $\ord_\mathfrak{p}(\gamma_i\delta_j-\gamma_j\delta_i)=0$.  Similarly, we find that $\ord_\mathfrak{p}(\gamma_i-\delta_i)=0$.  So $\ord_\mathfrak{p}\disc(F)=0$ as desired.

Now suppose that $\ord_{\mathfrak{p}}(x_P-x_Q)<0$.  Since $x_Q,\alpha_1,\ldots,\alpha_{2g+1}\in \O_{k,S}$, this implies that $\ord_{\mathfrak{p}}(x_P-x_Q)=\ord_{\mathfrak{p}}x_P=\ord_{\mathfrak{p}}(x_P-\alpha_i)<0$ for all $i$.  Since $x_Q-\alpha_i\in \O_{k,S}^*$ for all $i$, we have $\ord_\mathfrak{p}\gamma_i=-\frac{1}{2}\ord_{\mathfrak{p}}x_P$ for all $i$.  Let $c\in \O_{k,S}$ be such that $\ord_\mathfrak{p}c=\max\{0,-\frac{1}{2}\ord_\mathfrak{p}x_P\}$ for $\mathfrak{p}\not\in S$.  We consider now the binary form
\begin{equation*}
G(X,Z)=(cX-Z)\prod_{i=1}^{2g+1}(\delta_iX-\frac{\gamma_i}{c}Z)\in \O_{k,S}[X,Z].
\end{equation*}
The identity \eqref{bid} easily implies that if $\mathfrak{p}\not\in S$ and $\ord_{\mathfrak{p}}(x_P-x_Q)<0$, then
\begin{equation*}
\ord_\mathfrak{p}(\gamma_i\delta_j-\gamma_j\delta_i)=-\frac{1}{2}\ord_\mathfrak{p}x_P=\ord_\mathfrak{p}c, \quad i\neq j.
\end{equation*}
Then computing $\disc(G)$, we find that $\ord_\mathfrak{p}\disc(G)=0$ if $\mathfrak{p}\not\in S$ and $\ord_{\mathfrak{p}}(x_P-x_Q)\leq 0$.

Finally, suppose that $\ord_\mathfrak{p}(x_P-x_Q)>0$.  Then from \eqref{meq}, we must have $\ord_\mathfrak{p}(y_P-y_Q)= 0$ (the case $\ord_\mathfrak{p}(y_P-y_Q)< 0$ being impossible).  Since $\ord_\mathfrak{p}(x_Q-\alpha_i)=0$ for all $i$, $\ord_\mathfrak{p}(x_P-x_Q)>0$ implies that $\ord_\mathfrak{p}(x_P-\alpha_i)=0$ for all $i$.  Then $\ord_\mathfrak{p}\beta_i=\ord_\mathfrak{p}\delta_i=\ord_\mathfrak{p}\gamma_i=0$ for all $i$.  It follows from \eqref{bid} that
\begin{equation*}
\ord_\mathfrak{p}(\gamma_i^2\delta_j^2-\gamma_j^2\delta_i^2)=\ord_\mathfrak{p}(x_P-x_Q)
\end{equation*}
for $i\neq j$.  Similarly,
\begin{equation*}
\ord_\mathfrak{p}(\gamma_i^2-\delta_i^2)=\ord_\mathfrak{p}(x_P-x_Q)>0
\end{equation*}
for all $i$.  In particular, $\gamma_i\equiv \pm \delta_i\pmod{\mathfrak{p}}$ and $\beta_i\equiv \pm 1 \pmod{\mathfrak{p}}$ for all $i$.  Since  $x_P\equiv x_Q\pmod{\mathfrak{p}}$ and $y_P\not\equiv y_Q\pmod{\mathfrak{p}}$, we have $y_P\equiv -y_Q\pmod{\mathfrak{p}}$.  So 
\begin{equation*}
\prod_{i=1}^{2g+1}\beta_i\equiv \frac{y_Q}{y_P}\equiv -1\pmod{\mathfrak{p}}.
\end{equation*}
Then $\beta_i\equiv -1\pmod{\mathfrak{p}}$ for an odd number $n_{\mathfrak{p}}$ of the elements $i$.  Let $m_{\mathfrak{p}}=\ord_\mathfrak{p}(x_P-x_Q)$.  Then for $i\neq j$, 
\begin{equation*}
\ord_\mathfrak{p}(\gamma_i\delta_j-\gamma_j\delta_i)=
\begin{cases}
m_{\mathfrak{p}} \quad &\text{ if $\beta_i\equiv \beta_j \pmod{\mathfrak{p}}$},\\
0 \quad &\text{ if $\beta_i\not\equiv \beta_j \pmod{\mathfrak{p}}$}.
\end{cases}
\end{equation*}
Now a straight-forward calculation gives
\begin{equation}
\label{oeq}
\ord_\mathfrak{p}\disc(G)=m_{\mathfrak{p}}n_{\mathfrak{p}}(n_{\mathfrak{p}}-1)+m_{\mathfrak{p}}(2g+2-n_{\mathfrak{p}})(2g+2-n_{\mathfrak{p}}-1).
\end{equation}
Let 
\begin{equation*}
\mathcal{P}=\{\mathfrak{p}\in M_k\setminus S\mid \ord_\mathfrak{p}(x_P-x_Q)>0\}=\{\mathfrak{p}\in M_k\setminus S\mid \ord_\mathfrak{p}\disc(G)>0\}.  
\end{equation*}
Let $b\in \O_k$ be such that
\begin{equation*}
2bc\equiv -1\left(\bmod{ \prod_{\mathfrak{p}\in\mathcal{P}}\mathfrak{p}^{m_\mathfrak{p}}}\right).
\end{equation*}
Let $U=\left( \begin{array}{cc}
1 & b  \\
c & 1+bc
\end{array} \right)$.  Then
\begin{equation*}
-G_U(X,Z)=Z\prod_{i=1}^{2g+1}\left((\delta_i-\gamma_i)X+\frac{1}{c}(bc\delta_i-(bc+1)\gamma_i)Z\right).
\end{equation*}
Note that since $\det U=1$, $\disc(G_U)=\disc(G)$.  For $\mathfrak{p}\in \mathcal{P}$, let $\pi_\mathfrak{p}$ be a generator for $\mathfrak{p}\O_{k,S}$.  Let $\epsilon_{i,\mathfrak{p}}=1$ if $\delta_i\equiv \gamma_i\pmod{\mathfrak{p}}$ and $\epsilon_{i,\mathfrak{p}}=0$ otherwise (in which case $\delta_i\equiv -\gamma_i\pmod{\mathfrak{p}}$).  Define $\theta_i=\prod_{\mathfrak{p}\in \mathcal{P}}\pi_{\mathfrak{p}}^{m_\mathfrak{p}\epsilon_{i,\mathfrak{p}}}$ and $\theta_i'=\prod_{\mathfrak{p}\in \mathcal{P}}\pi_{\mathfrak{p}}^{m_\mathfrak{p}(1-\epsilon_{i,\mathfrak{p}})}$.  Consider the binary form
\begin{equation*}
H(X,Z)=Z\prod_{i=1}^{2g+1}\left(\frac{\delta_i-\gamma_i}{\theta_i}X+\frac{\theta_i'}{c}(bc\delta_i-(bc+1)\gamma_i)Z\right).
\end{equation*}
Note that the binary form $H(X,Z)$ is a scalar multiple of $G_U\left(X,\left(\prod_{\mathfrak{p}\in \mathcal{P}}\pi_{\mathfrak{p}}^{m_\mathfrak{p}}\right)Z\right)$.  It follows that $\ord_\mathfrak{p}\disc(H)=0$ if $\mathfrak{p}\in M_k\setminus (\mathcal{P}\cup S)$.  For $\mathfrak{p}\in \mathcal{P}$, from \eqref{oeq} and the definition of $H$, a calculation yields
\begin{equation}
\label{deq}
\ord_\mathfrak{p}\disc(H)=2m_{\mathfrak{p}}n_{\mathfrak{p}}(n_{\mathfrak{p}}-1).
\end{equation}
For $\mathfrak{p}\in \mathcal{P}$,
\begin{equation*}
2(bc\delta_i-(bc+1)\gamma_i)\equiv -(\delta_i+\gamma_i)\pmod{\mathfrak{p}^{m_\mathfrak{p}}}.
\end{equation*}
If $\epsilon_{i,\mathfrak{p}}=0$, then it follows that
$\ord_\mathfrak{p}(\frac{\theta_i'}{c}(bc\delta_i-(bc+1)\gamma_i))\geq 2m_\mathfrak{p}$.  In fact, since there are $n_\mathfrak{p}$ distinct values of $i$ such that $\epsilon_{i,\mathfrak{p}}=0$, by \eqref{deq}, $\ord_\mathfrak{p}(\frac{\theta_i'}{c}(bc\delta_i-(bc+1)\gamma_i))= 2m_\mathfrak{p}$ for at most one value of $i$ with $\epsilon_{i,\mathfrak{p}}=0$.  After an appropriate substitution $X\mapsto X+aZ$, we can force $\ord_\mathfrak{p}(\frac{\theta_i'}{c}(bc\delta_i-(bc+1)\gamma_i))= 2m_\mathfrak{p}$ for all $i$ and $\mathfrak{p}$ such that $\epsilon_{i,\mathfrak{p}}=0$.  Then for each $\mathfrak{p}\in \mathcal{P}$, $H(x,1)$ has $n_\mathfrak{p}$ roots $\alpha$ with $\ord_\mathfrak{p}\alpha=2m_\mathfrak{p}$.  If $n_\mathfrak{p}=2g+1$, then we can replace $H(X,Z)$ by $\pi_\mathfrak{p}^{m_\mathfrak{p}}H\left(X,\frac{Z}{\pi_\mathfrak{p}^{m_\mathfrak{p}}}\right)$, eliminating $\mathfrak{p}$ as a divisor of the discriminant of $H$.  So we have $3\leq n_\mathfrak{p}\leq 2g-1$ for every $\mathfrak{p}\in M_k\setminus S$ with $\ord_{\mathfrak{p}}\disc(H)>0$.  Then from all of the above we have found a binary form equivalent to $(X-Z)\prod_{i=1}^{2g+1}(X-\beta_iZ)$ showing that $[(X-Z)\prod_{i=1}^{2g+1}(X-\beta_iZ)]\in \mathcal{B}'(2g+2,k,S)$.
\end{proof}

We have shown that in order to solve Problem \ref{mp} it suffices to effectively compute, for all values of the parameters, either the set $\mathcal{B}'(2g+2,k,S)$ or the set $\mathcal{H}'(g,k,S)$.  It seems interesting to determine the precise relationship between these two sets.  In the case $g=2$, Liu \cite{Liu2} has given an algorithm to compute, given a Weierstrass equation for a genus two curve, the fibers of a minimal model of the curve (away from primes above $2$, at least).  In particular (see \cite[\S 6]{Liu2}), Liu's algorithm implies that computing the set $\mathcal{B}'(6,k,S)$ for all $k$ and $S$ and computing the set $\mathcal{H}'(2,k,S)$ for all $k$ and $S$ are equivalent problems.  More precisely, Liu's results imply the following relationship.

\begin{theorem}
\label{teq}
Let $k$ be a number field and $S$ a finite set of places of $k$ containing the archimedean places and the places lying above $2$.  There exists a number field $L$ and a finite set of places $T$ of $L$ such that if $C:y^2=f(x)$, $\deg f=6$, is a hyperelliptic curve of genus two representing an element of $\mathcal{H}'(2,k,S)$ and $F(X,Z)$ is the homogenization of $f$, then $[F]\in \mathcal{B}'(6,L,T)$.  Conversely, if $[F]\in \mathcal{B}'(6,k,S)$, then for some constant $c\in k^*$, the equivalence class of the curve $y^2=cF(x,1)$ is in $\mathcal{H}'(2,k,S)$.
\end{theorem}

In the theorem, we can take $L$ to be any field such that for every curve $C$ representing an element of $\mathcal{H}'(2,k,S)$, the Weierstrass points of $C$ are $L$-rational.  As is well known, such a field $L$ can be explicitly computed and it is a certain extension of $k$ unramified outside of $S$.

It seems plausible that the analogue of Theorem \ref{teq} holds in higher genus.  We give a brief proof of one of the directions.

\begin{theorem}
\label{gr}
Let $g$ be a positive integer, $k$ be a number field, and $S$ a finite set of places of $k$ containing the archimedean places and the places lying above $2$.  If $[F]\in \mathcal{B}'(2g+2,k,S)$, then for some constant $c\in k^*$ the equivalence class of the curve $y^2=cF(x,1)$ is in $\mathcal{H}'(g,k,S)$.  
\end{theorem}

\begin{proof}
Let $[F]\in \mathcal{B}'(2g+2,k,S)$, with $F\in \O_{k,S}[X,Z]$ as in the definition of $\mathcal{B}'(2g+2,k,S)$.  Let $f(x)=F(x,1)$ and let $C$ be the hyperelliptic curve defined by $y^2=f(x)$.  It suffices to show that $\Jac(C)$ has good reduction outside $S$.  Let $\mathfrak{p}\in M_k\setminus S$.  If $\mathfrak{p}$ doesn't divide the discriminant of $f$ then the curve $C$, and hence $\Jac(C)$, has good reduction at $\mathfrak{p}$.  Otherwise, let $k_\mathfrak{p}$ be the completion of $k$ at $\mathfrak{p}$, $\O_\mathfrak{p}$ the ring of integers of $k_\mathfrak{p}$, and $\pi$ a uniformizer.  Then from the definitions, we can write 
\begin{equation*}
f(x)=h(x)\prod_{i=1}^n(x-u_i\pi^{2m}),
\end{equation*}
where
\begin{itemize}
\item  $m$ and $n$ are positive integers with $n$ odd, $3\leq n\leq 2g-1$.
\item  $h\in \O_\mathfrak{p}[x]$, $h \pmod{\pi}$ has distinct roots, and $\pi\nmid h(0)$.
\item  $u_i\in \O_\mathfrak{p}^*$, $i=1,\ldots, n$, and $\prod_{i=1}^n(x-u_i)$ has distinct roots mod $\pi$.
\end{itemize}
Note that $C$ can also be defined, over $k_\mathfrak{p}$, by the equation $y^2=h(\pi^{2m}x)\prod_{i=1}^n(x-u_i)$.  Let $C_1$ and $C_2$ be the hyperelliptic curves over $\O_k/\mathfrak{p}$ defined by $y^2=x\overline{h}(x)$ and $y^2=\overline{h(0)}\prod_{i=1}^n(x-\overline{u}_i)$, respectively, where the bar denotes the image in $\left(\O_k/\mathfrak{p}\right)[x]$.  Then one can show that the special fiber of the minimal proper regular model of $C$ over $\O_\mathfrak{p}$ consists of either $C_1$ and $C_2$ intersecting at a single point or $C_1$, $C_2$, and a chain of rational curves joining a point of $C_1$ with a point of $C_2$.  Then it follows from well-known facts on the relationship between a minimal model of $C$ and the N\'eron model of $\Jac(C)$ \cite[\S 9.5 Th. 4, \S 9.6 Prop. 10]{BLR} that $\Jac(C)$ has good reduction at $\mathfrak{p}$ and the reduction mod $\mathfrak{p}$ is isomorphic to $\Jac(C_1)\times \Jac(C_2)$.
\end{proof}

Finally, since $\mathcal{H}'(g,k,S)$ is finite by the Shafarevich conjecture, we can conclude the (ineffective) finiteness of the set $\mathcal{B}'(2g+2,k,S)$.

\begin{corollary}
\label{SF}
For any positive integer $g$, number field $k$, and finite set of places $S$ of $k$, the set $\mathcal{B}'(2g+2,k,S)$ is finite.
\end{corollary}

\subsection*{Acknowledgments}
The author would like to thank Yuri Bilu, Rafael von K\"anel, and Umberto Zannier for several helpful comments on an earlier draft of the paper.
\providecommand{\bysame}{\leavevmode\hbox to3em{\hrulefill}\thinspace}
\providecommand{\MR}{\relax\ifhmode\unskip\space\fi MR }
\providecommand{\MRhref}[2]{%
  \href{http://www.ams.org/mathscinet-getitem?mr=#1}{#2}
}
\providecommand{\href}[2]{#2}

\end{document}